\documentclass{amsart}
\usepackage[utf8]{inputenc}
\usepackage{amsmath}
\usepackage{amsfonts}
\usepackage{amssymb}
\usepackage{amsthm}
\usepackage{mathtools}	% This is amsmath+additional tools
\usepackage{tikz}       % For drawing pictures
\usepackage{tikz-cd}	% For commutative diagrams 
\usepackage{hyperref}	% For links
\usepackage{enumitem}
\usepackage{graphicx}

\newcommand{\F}{\ensuremath{\mathbb{F}} }
\newcommand{\N}{\ensuremath{\mathbb{N}} }
\newcommand{\Z}{\ensuremath{\mathbb{Z}} }

\newcommand{\CFK}{\mathit{CFK}}
\newcommand{\HFK}{\mathit{HFK}}

\newtheorem{theorem}{Theorem}[section]

\newtheorem{proposition}[theorem]{Proposition}

\newtheorem{lemma}[theorem]{Lemma}
\newtheorem{corollary}[theorem]{Corollary}

\theoremstyle{definition}
\newtheorem{definition}[theorem]{Definition}

\newtheorem{algorithm}[theorem]{Algorithm}

\theoremstyle{remark}
\newtheorem*{remark}{Remark}
\newtheorem{example}[theorem]{Example}

\usepackage{subcaption}

\newcommand{\gr}[0]{\text{gr}}

\colorlet{dark green}{green!50!black}

\title{Algebraic realizability of knot Floer-like complexes}
\author{David Popovi\'c}
\address{Department of Mathematics\\
  University of California, Los Angeles}
\email{dpopovic@math.ucla.edu}

\date{January 2023}

\begin{document}
\maketitle
\begin{abstract}
    We study algebraic obstructions to realizability of local equivalence classes of knot-like complexes. We classify local equivalence classes of knot-like complexes over $\F[U,V]$, answering a question of Dai, Hom, Stoffregen and Truong.
\end{abstract}

\section{Introduction}
Knot Floer homology is an invariant of knots and links in 3-manifolds. It was introduced independently by Ozsv\'ath and Szab\'o \cite{ozsvath2004holomorphicKnots} and Rasmussen \cite{rasmussen2003floer} as a refinement of an earlier theory of $3$-manifold invariants Heegaard Floer homology \cite{ozsvath2004holomorphic, ozsvath2004holomorphicSequel}. Since then, different variations of the knot Floer homology package have been found to contain much information about the geometric properties of the knot. For example, knot Floer homology detects the genus of a knot \cite{OS2004holomorphicDisksAndGenusBounds}, its fiberedness \cite{ghiggini2008knot, ni2007knot, juhasz2008floer}, the Thurston norm of a knot complement \cite{OS2008linkThurstonNorm} and has been used to obtain bounds on the unknotting number \cite{OS-HFKandUnknottingNumber} and slice genus \cite{ozsvath2003tau}. 

\vspace{1em}

One area in which the strength of knot Floer homology has been leveraged with particular success has been the study of the smooth knot concordance group $\mathcal{C}$. See \cite{hom2017survey} for a survey article on the subject. Numerous knot concordance invariants have been constructed, for example $\tau$, $\nu$, $\varepsilon$, $\nu^+$, $V_s$, $\overline{V}_0$, $\underline{V}_0$, $\Upsilon(t)$, and most recently $\phi_j$ \cite{ozsvath2003tau, OS-HFKandRationalSurgeries, hom2014epsilon, hom2014nu+, OS-HFKandIntegerSurgeries, hendricks2017involutive, OSS2017upsilon, dai2021more}. Over the years, this development has crystallized the importance of the notion of local equivalence; most of the mentioned invariants factor through the local equivalence group of knot Floer complexes (see the local equivalence in \cite{zemke2019connected} without the involutive part, stable equivalence in \cite{hom2017survey}, or $\nu^+$-equivalence in \cite{kim2016infinite} for some different descriptions of the same concept). 

\vspace{1em}

In their recent paper introducing an infinite family of linearly independent homomorphisms $\phi_j: \mathcal{C} \to \Z$  \cite{dai2021more}, Dai, Hom, Stoffregen, and Truong studied the simplified knot Floer complexes $\CFK_{\mathcal{R}_1}(K)$ over the ring $\mathcal{R}_1=\frac{\F[U,V]}{(UV)}$. The main advantage of setting $UV=0$ is that the resulting local equivalence group of knot Floer complexes over $\mathcal{R}_1$ becomes totally ordered. In turn, this yields a simple classification of knot Floer complexes up to local equivalence in which every class can be described by an even length finite sequence of nonzero integers.

\vspace{1em}

However, not all such sequences correspond to knots. It remains an interesting and difficult question to determine which local equivalence classes can be realized by $\CFK_{\mathcal{R}_1}(K)$ for a knot $K \subset S^3$. Since $\CFK_{\mathcal{R}_1}(K)$ is a mod $UV$ reduction of a chain complex $\CFK_{\F[U,V]}(K)$ over $\F[U,V]$, every realized class should have such a representative. This leads the authors of \cite{dai2021more} to pose  the following purely algebraic question: Which local equivalence classes arise as a mod $UV$ reduction of some chain complex over $\F[U,V]$? A complete classification of these classes is the main result of this paper.
\begin{theorem} \label{theorem}
Let $a_1, \dots, a_{2n}$ be a sequence of nonzero integers. The local equivalence class of $C(a_1, \dots, a_{2n})$ (Def. \ref{def:standard complex}) is algebraically realizable (Def. \ref{def:algebraically realizable}) if and only if the complex $C(a_1, \dots, a_{2n})$ is partially realizable (Def. \ref{def:partially and fully realizable}). The latter can be decided by Algorithm \ref{algorithm} terminating in at most $n^2+n$ steps.

As a special case, the local equivalence class of $C(a_1, \dots, a_{2n})$ is algebraically realizable if $|a_i| \geq 2$ for all $i$.
\end{theorem}
We stress that the algorithm we designed is very simple. It is geometric in nature, requires no algebraic calculations, and can be carried out for reasonably-sized complexes by hand in a matter of seconds. We give some demonstrations of the algorithm in Section \ref{section:Examples}.

\vspace{1em}

It will follow from the definition of algebraic realizability that the chain complexes realizing local equivalence classes are not just chain complexes over $\F[U,V]$; they share many other properties with knot Floer complexes $\CFK_{\F[U,V]}(K)$. The author is not presently aware of any symmetric algebraically realizable class that does not arise from a knot.

\subsection*{Acknowledgements} I would like to thank Sucharit Sarkar for introducing me to Heegaard Floer homology and for his guidance. This work was partially supported by NSF Grant DMS-1905717.

\section{Background and previous work}\label{Background Section}
The purpose of this section is to establish our notational conventions and give an overview of previous work. Throughout this paper, we work over $\F=\Z/2\Z$.
\subsection{Knot Floer complexes}\label{subsection: Knot Floer complexes}
We begin by reviewing formal aspects of knot Floer homology. The main reference that uses the same terminology is \cite{dai2021more} and the original paper containing the proofs of the properties is \cite{ozsvath2004holomorphicKnots}. 

Let $K \subset S^3$ be a knot. The knot Floer complexes $\CFK_{\F[U,V]}(K)$ have the following algebraic properties.
\begin{enumerate}
    \item \textbf{Gradings}: $\CFK_{\F[U,V]}(K)$ is a finitely generated bigraded chain complex over $\F[U,V]$ with a differential $\partial$ and bigrading $\gr=(\gr_U, \gr_V)$ satisfying $\gr(U)=(-2,0)$, $\gr(V)=(0,-2)$, and $\gr(\partial)=(-1,-1)$.
    \item \textbf{Symmetry}: There is a chain homotopy equivalence
    $$\CFK_{\F[U,V]}(K)\simeq \overline{\CFK_{\F[U,V]}}(K)$$
    where the overline denotes the complex in which the roles of $U$ and $V$ are exchanged and $\gr_U$ and $\gr_V$ are switched. %The chain homotopy equivalence comes from where?
    \item \textbf{Homology}: There is an isomorphism
    $$\frac{H_*(\CFK_{\F[U,V]}(K)/U)}{V\text{-torsion}} \cong \F[V]$$
    as bigraded chain complexes over $\F[V]$ where $V$-torsion denotes the torsion subcomplex. The isomorphism stems from the fact that we are considering knots in $S^3$. Setting $U=0$ and taking homology recovers $\HFK^-(S^3, K)$ whose nontorsion part is $\F[V]$, generated by an element $x$ with $\gr_U(x)=0$. There is a similar isomorphism obtained by exchanging the roles of $U$ and $V$.
\end{enumerate}
Frequently, the easiest way to describe knot Floer complexes and their algebraic counterparts defined in the next section is to draw them. The knot Floer generators over $\F$ are drawn on a $2$-dimensional lattice where the $x$-coordinate denotes the $U^{-1}$ power and the $y$-coordinate denotes the $V^{-1}$ power. The arrows connecting the generators indicate the differential $\partial$, as explained below. Since $\gr(\partial)=(-1, -1)$, the Alexander grading $A=\frac{1}{2}(\gr_U-\gr_V)$ is preserved by $\partial$ and $\CFK_{F[U,V]}(K)$ splits as an $\F[UV]$-module (but \emph{not} as an $\F[U,V]$-module) as a direct sum of summands with a constant Alexander grading. Since the pictures corresponding to different summands are all translates of each other, we always only draw the part in the Alexander grading $A=0$ and are intentionally ambiguous about the origin of the plane. Note that even after this restriction, the same letter will appear in infinitely many places. For example, in Figure \ref{fig:infinite CFK(trefoil)}, the generator $c$ is drawn three times in positions $(-1,0)$, $(-2, -1)$, and $(-3, -2)$, and the three locations denote the elements $Uc$, $U^2Vc$, and $U^3V^2c$. The top horizontal arrow between $a$ and $b$ indicates that $\partial Va = UVb$ or equivalently that $\partial a = Ub$.

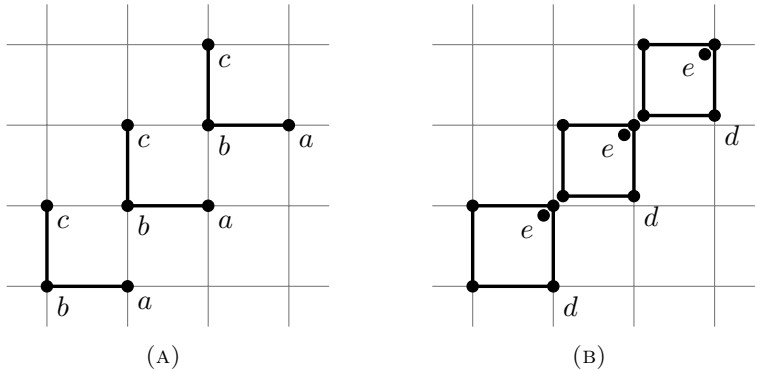
\begin{figure}[t]
    \begin{subfigure}[b]{0.35\textwidth}
        \centering
        \resizebox{\linewidth}{!}{
        \begin{tikzpicture}[scale=1.0]
            \draw[step=1.0,gray,thin] (-0.5,-0.5) grid (3.5,3.5);
            \draw[black, very thick] (1,0) -- (0,0) -- (0,1);
            \draw[black, very thick] (2,1) -- (1,1) -- (1,2);
            \draw[black, very thick] (3,2) -- (2,2) -- (2,3);
            
            \filldraw[black] (0,0) circle (2pt) node[anchor=north west]{$b$};
            \filldraw[black] (1,0) circle (2pt) node[anchor=north west]{$a$};
            \filldraw[black] (0,1) circle (2pt) node[anchor=north west]{$c$};

            \filldraw[black] (1,1) circle (2pt) node[anchor=north west]{$b$};
            \filldraw[black] (2,1) circle (2pt) node[anchor=north west]{$a$};
            \filldraw[black] (1,2) circle (2pt) node[anchor=north west]{$c$};

            \filldraw[black] (2,2) circle (2pt) node[anchor=north west]{$b$};
            \filldraw[black] (3,2) circle (2pt) node[anchor=north west]{$a$};
            \filldraw[black] (2,3) circle (2pt) node[anchor=north west]{$c$};
        \end{tikzpicture}}
        \caption{}
        \label{fig:infinite CFK(trefoil)}
    \end{subfigure}
    \hspace{1cm}
    \begin{subfigure}[b]{0.35\textwidth}
        \centering
        \resizebox{\linewidth}{!}{
        \begin{tikzpicture}[scale=1.0]
            \def\a{0.12}
            \def\b{0.17}
            \draw[step=1.0,gray,thin] (-0.5,-0.5) grid (3.5,3.5);
            \draw[black, very thick] (1,0) -- (0,0) -- (0,1) -- (1,1) -- (1,0);
            \draw[black, very thick] (2,1+\a) -- (1+\a,1+\a) -- (1+\a,2) -- (2,2) -- (2,1+\a);
            \draw[black, very thick] (3,2+\a) -- (2+\a,2+\a) -- (2+\a,3) -- (3,3) -- (3,2+\a);
            
            \filldraw[black] (0,0) circle (2pt) node[anchor=north west]{};
            \filldraw[black] (1,0) circle (2pt) node[anchor=north west]{$d$};
            \filldraw[black] (0,1) circle (2pt) node[anchor=north west]{};
            \filldraw[black] (1,1) circle (2pt) node[anchor=north west]{};
            \filldraw[black] (1-\a, 1-\a) circle (2pt) node[anchor=north east]{$e$};

            \filldraw[black] (1+\a,1+\a) circle (2pt) node[anchor=north west]{};
            \filldraw[black] (2,1+\a) circle (2pt) node[anchor=north west]{$d$};
            \filldraw[black] (1+\a,2) circle (2pt) node[anchor=north west]{};
            \filldraw[black] (2,2) circle (2pt) node[anchor=north west]{};
            \filldraw[black] (2-\a,2-\a) circle (2pt) node[anchor=north east]{$e$};

            \filldraw[black] (2+\a,2+\a) circle (2pt) node[anchor=north west]{};
            \filldraw[black] (3,2+\a) circle (2pt) node[anchor=north west]{$d$};
            \filldraw[black] (2+\a,3) circle (2pt) node[anchor=north west]{};
            \filldraw[black] (3,3) circle (2pt) node[anchor=north west]{};
            \filldraw[black] (3-\a,3-\a) circle (2pt) node[anchor=north east]{$e$};
        \end{tikzpicture}}
        \caption{}
        \label{fig:infinite CFK(figure 8 knot)}
    \end{subfigure}
    \caption{The knot Floer complexes of the negative trefoil ($\textsc{a}$) and the figure-eight knot ($\textsc{b}$) with some generators labelled.} 
    \label{fig:infinite CFK}
\end{figure}

The knot Floer complexes of the negative trefoil and the figure-eight knot drawn in Figure \ref{fig:infinite CFK} consist of a finite piece and $\Z$ many of its translations. As such, it is sufficient to only draw the finite piece, which by itself uniquely specifies the chain complex. The corresponding finite pictures of the negative trefoil and the figure-eight knot are shown in Figure \ref{fig:finite CFK}.

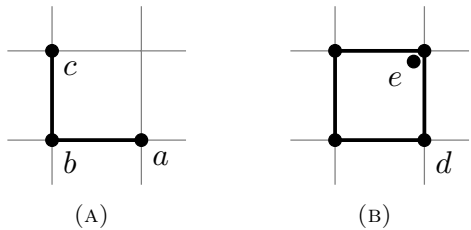
\begin{figure}[t]
    \begin{subfigure}[b]{0.20\textwidth}
        \centering
        \resizebox{\linewidth}{!}{
        \begin{tikzpicture}[scale=1.0]
            \draw[step=1.0,gray,thin] (-0.5,-0.5) grid (1.5,1.5);
            \draw[black, very thick] (1,0) -- (0,0) -- (0,1);
            \filldraw[black] (0,0) circle (2pt) node[anchor=north west]{$b$};
            \filldraw[black] (1,0) circle (2pt) node[anchor=north west]{$a$};
            \filldraw[black] (0,1) circle (2pt) node[anchor=north west]{$c$};
        \end{tikzpicture}}
        \caption{}
        \label{fig:CFK(trefoil)}
    \end{subfigure}
    \hspace{1cm}
    \begin{subfigure}[b]{0.20\textwidth}
        \centering
        \resizebox{\linewidth}{!}{
        \begin{tikzpicture}[scale=1.0]
            \draw[step=1.0,gray,thin] (-0.5,-0.5) grid (1.5,1.5);
            \def\a{0.12}
            \draw[black, very thick] (1,0) -- (0,0) -- (0,1) -- (1,1) -- (1,0);
            \filldraw[black] (0,0) circle (2pt) node[anchor=north west]{};
            \filldraw[black] (1,0) circle (2pt) node[anchor=north west]{$d$};
            \filldraw[black] (0,1) circle (2pt) node[anchor=north west]{};
            \filldraw[black] (1,1) circle (2pt) node[anchor=north west]{};
            \filldraw[black] (1-\a,1-\a) circle (2pt) node[anchor=north east]{$e$};
        \end{tikzpicture}}
        \caption{}
        \label{fig:CFK(figure 8 knot)}
    \end{subfigure}
    \caption{The finite pieces of the knot Floer complexes of the negative trefoil ($\textsc{a}$) and the figure-eight knot ($\textsc{b}$).}
    \label{fig:finite CFK}
\end{figure}

In general, however, it can happen that the complex is \emph{essentially infinite}, that is, it does not split into finite pieces. In such cases, we draw pictures large enough to completely specify the differential $\partial$. These may sometimes contain the same label twice, for example in Figure \ref{fig:C(2,2)}, but there is no reason to be alarmed. We emphasize that the two $z$'s appearing in that picture represent different elements $z$ and $UVz$ and there is \emph{a priori} no reason why these would not be connected in the knot Floer complex. This phenomenon is discussed in more detail in \cite{popovic2023link}, where the author constructs algebraic complexes sharing many properties with $\CFK_{\F[U,V]}(K)$ which remain infinite after any change of basis or even chain homotopy equivalence.

While the knot Floer complexes of the trefoil and the figure-eight knot depicted in Figure \ref{fig:infinite CFK} and Figure \ref{fig:finite CFK} only contain horizontal and vertical arrows, diagonal arrows will also be present in general. In that case, setting $UV=0$ corresponds to deleting all diagonal arrows and setting $U^2V^2=0$ corresponds to deleting diagonal arrows that move in both directions by at least $2$.

Finally, we note that our pictures are equivalent to the pictures obtained in a more classical setting of knot Floer homology \cite{ozsvath2004holomorphicKnots}. To be precise; the generators have different decorations by monomials, but the shapes are the same as the shapes obtained upon consideration of knot Floer complexes $\CFK^-(K)$ over $\F[U]$ in which the $x$-axis represents the $U^{-1}$ power and the $y$-axis represents the Alexander grading $A$.

\subsection{Algebraic complexes}
With the formal properties of $\CFK_{F[U,V]}(K)$ in mind, we establish an algebraic framework for dealing with such complexes.

\vspace{1em}

Let the ring $\F[U,V]$ be equipped with a relative $\Z \oplus \Z$ grading $\gr=(\gr_U, \gr_V)$, where $\gr(U) = (-2, 0)$ and $\gr(V)=(0, -2)$. Any quotient of $\F[U,V]$ by an ideal generated by homogeneous elements inherits the relative grading $\gr$ in a natural way. The cases we will be considering most often are when the quotient is $\mathcal{R}_i := \frac{\F[U,V]}{(U^iV^i)}$ for $i \in \N$ and when it is $\mathcal{R}_\infty := \F[U,V]$ itself. Any module $C$ over these relatively graded rings is automatically considered $\Z \oplus \Z$ graded and any endomorphism of $C$ that is denoted by $\partial$ is required to have degree $(-1, -1)$.
\begin{definition}
    Let $i \in \N\cup\{\infty\}$. A free $\mathcal{R}_i$-module with an endomorphism and a distinguished basis is a triple $(C, \partial, B)$ where $C$ is a free $\mathcal{R}_i$-module, $\partial: C \to C$ is an endomorphism, and $B$ is a basis of $C$.

    If $\partial^2=0$, we call such a triple a free chain complex over $\mathcal{R}_i$ with a distinguished basis.
\end{definition}
\begin{remark}
    We sometimes refer to $(C, \partial, B)$ just by $C$, leaving $\partial$ and $B$ implicit.
\end{remark}

Let $i, j \in \N\cup\{\infty\}$ be such that $i \leq j$. Note that any free chain complex over $\mathcal{R}_i$ with a distinguished basis $(C, \partial, B)$ can be lifted to a free $\mathcal{R}_j$-module with an endomorphism and a distinguished basis $(\widehat{C}, \widehat{\partial}, B)$ where $\widehat{C} = \mathcal{R}_j\langle B\rangle$ is the free $\mathcal{R}_j$-module generated by $B$ and $\widehat{\partial}$ is defined on $B$ in the same way as $\partial$. Since most of our work is concerned with lifting chain complexes over $\mathcal{R}_1$ to chain complexes over $\mathcal{R}_\infty = \F[U,V]$, sometimes passing through $\mathcal{R}_2$ on the way, this construction will be used throughout the paper. It will also be convenient to have the following definition.
\begin{definition}
    Let $i \in \N\cup\{\infty\}$ and let $(C, \partial, B)$ be a free $\mathcal{R}_i$-module with an endomorphism and a distinguished basis. For basis elements $x, y \in B$ and $a, b \in \N_0$, let $\langle \partial x, U^aV^by\rangle \in \F$ denote the coefficient of $U^aV^by$ in $\partial x$.
\end{definition}
We have so far introduced the necessary terminology for dealing with the `Gradings' part of the list in Subsection \ref{subsection: Knot Floer complexes}. Encoding the `Symmetry' and `Homology' properties of knot Floer complexes into a purely algebraic language is more straightforward. 
\begin{definition}
    Let $i\in \N \cup \{\infty\}$. A free $\mathcal{R}_i$-module with an endomorphism $(C, \partial)$ is \emph{symmetric} if $C \simeq \overline{C}$ where the overline denotes the complex in which the roles of $U$ and $V$ are interchanged and $\gr_U$ and $\gr_V$ are switched.
\end{definition}
\begin{definition}
    Let $i\in \N \cup \{\infty\}$. A free $\mathcal{R}_i$-module with an endomorphism $(C, \partial)$ has the \emph{correct homology} if $\frac{H_*(C/U)}{V-\text{torsion}} \cong \F[V]$ and $\frac{H_*(C/V)}{U-\text{torsion}} \cong \F[U]$, where the element $x$ generating $\F[V]$ satisfies $\gr_U(x)=0$ and the element $y$ generating $\F[U]$ satisfies $\gr_V(y)=0$.
\end{definition}

We now define knot-like complexes and standard complexes that were first introduced in \cite{dai2021more}.
\begin{definition}
    A \emph{knot-like complex} $(C, \partial)$ is a finitely generated free chain complex over $\mathcal{R}_1$ with the correct homology.
\end{definition}
Note that knot-like complexes are not required to be symmetric.
\begin{definition}\label{def:standard complex}
Let $n \in \N$ and let $a_1, \dots, a_{2n}$ be a sequence of nonzero integers. The \emph{standard complex} $C(a_1, \dots, a_{2n})$ is the free chain complex over $\mathcal{R}_1$ with a distinguished basis $(C, \partial, B)$ where $B=\{x_0, \dots, x_{2n}\}$ and the differential $\partial$ is as follows. For each odd $i$, there is a $U$-arrow of length $|a_i|$ connecting $x_i$ and $x_{i-1}$. For each even $i$, there is a $V$-arrow of length $|a_i|$ connecting $x_i$ and $x_{i-1}$. The direction of the arrow is determined by the sign of $a_i$, as follows. If $a_i > 0$, then the arrow goes from $x_i$ to $x_{i-1}$, and if $a_i < 0$, then the arrow goes from $x_{i-1}$ to $x_i$. Finally, we equip the distinguished basis elements $x_0, \dots, x_{2n}$ with the unique $\Z\oplus\Z$ gradings making $C(a_1, \dots, a_{2n})$ into a knot-like complex.
%All other differentials are zero.
\end{definition}
It will also be useful in our discussion to have the notion of a standard complex with two extra arrows adjacent to its endpoints. We call them extended standard complexes.
\begin{definition}
Let $n \in \N$ and let $a_0, \dots, a_{2n+1}$ be a sequence of nonzero integers. The \emph{extended standard complex} $C(a_0 \ | \ a_1, \dots, a_{2n} \ | \ a_{2n+1})$ is the free chain complex over $\mathcal{R}_1$ with a distinguished basis $(C, \partial, B)$ where $B=\{x_{-1}, \dots, x_{2n+1}\}$ and the differential $\partial$ is as follows. For each odd $i$, there is a $U$-arrow of length $|a_i|$ connecting $x_i$ and $x_{i-1}$. For each even $i$, there is a $V$-arrow of length $|a_i|$ connecting $x_i$ and $x_{i-1}$. The direction of the arrow is determined by the sign of $a_i$, as follows. If $a_i > 0$, then the arrow goes from $x_i$ to $x_{i-1}$, and if $a_i < 0$, then the arrow goes from $x_{i-1}$ to $x_i$. The $\Z \oplus \Z$ gradings of the distinguished basis elements $x_0, \dots, x_{2n}$ are the same as in $C(a_1, \dots, a_{2n})$ and the $\Z\oplus\Z$ gradings of $x_{-1}$ and $x_{2n+1}$ are then uniquely determined by the requirement that $\partial$ have degree $(-1,-1)$. 
\end{definition}

For understanding the notions of standard and extended standard complexes, Figure \ref{fig:examples of std and ext std cxs} is likely to be a lot more illuminating than the preceding definitions.
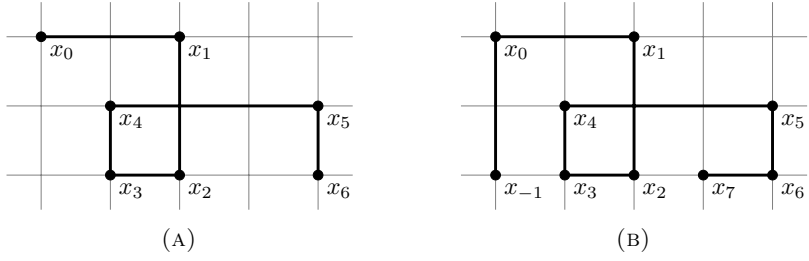
\begin{figure}[t]
    \begin{subfigure}[b]{0.38\textwidth}
        \centering
        \resizebox{\linewidth}{!}{
        \begin{tikzpicture}[scale=1.0]
            \draw[step=1.0,gray,thin] (-0.5,1.5) grid (4.5,4.5);
            \draw[black, very thick] (0,4) -- (2,4) -- (2,2) -- (1,2) -- (1,3) -- (4,3) -- (4,2);
            \filldraw[black] (0,4) circle (2pt) node[anchor=north west]{$x_0$};
            \filldraw[black] (2,4) circle (2pt) node[anchor=north west]{$x_1$};
            \filldraw[black] (2,2) circle (2pt) node[anchor=north west]{$x_2$};
            \filldraw[black] (1,2) circle (2pt) node[anchor=north west]{$x_3$};
            \filldraw[black] (1,3) circle (2pt) node[anchor=north west]{$x_4$};
            \filldraw[black] (4,3) circle (2pt) node[anchor=north west]{$x_5$};
            \filldraw[black] (4,2) circle (2pt) node[anchor=north west]{$x_6$};
        \end{tikzpicture}}
        \caption{}
        \label{fig:example of a standard complex}
    \end{subfigure}
    \hspace{1cm}
    \begin{subfigure}[b]{0.38\textwidth}
        \centering
        \resizebox{\linewidth}{!}{
        \begin{tikzpicture}[scale=1.0]
            \draw[step=1.0,gray,thin] (-0.5,1.5) grid (4.5,4.5);
            \draw[black, very thick] (0,2) -- (0,4) -- (2,4) -- (2,2) -- (1,2) -- (1,3) -- (4,3) -- (4,2) -- (3,2);
            \filldraw[black] (0,2) circle (2pt) node[anchor=north west]{$x_{-1}$};
            \filldraw[black] (0,4) circle (2pt) node[anchor=north west]{$x_0$};
            \filldraw[black] (2,4) circle (2pt) node[anchor=north west]{$x_1$};
            \filldraw[black] (2,2) circle (2pt) node[anchor=north west]{$x_2$};
            \filldraw[black] (1,2) circle (2pt) node[anchor=north west]{$x_3$};
            \filldraw[black] (1,3) circle (2pt) node[anchor=north west]{$x_4$};
            \filldraw[black] (4,3) circle (2pt) node[anchor=north west]{$x_5$};
            \filldraw[black] (4,2) circle (2pt) node[anchor=north west]{$x_6$};
            \filldraw[black] (3,2) circle (2pt) node[anchor=north west]{$x_7$};
        \end{tikzpicture}}
        \caption{}
        \label{fig:example of an extended standard complex}
    \end{subfigure}
    \caption{Standard complex $C(2, -2, -1, 1, 3, -1)$ in ($\textsc{a}$) and an extended standard complex $C(2\ | \ 2, -2, -1, 1, 3, -1 \ |-1)$ in ($\textsc{b}$).}
    \label{fig:examples of std and ext std cxs}
\end{figure}
Note that any standard complex is a knot-like complex, but extended standard complexes are not knot-like complexes, since their homology is not correct.

\vspace{1em}

Standard complexes are interesting primarily because they classify local equivalence types of knot-like complexes.
\begin{theorem}[\cite{dai2021more}, Corollary 6.2]\label{classification}
Let $C$ be a knot-like complex. Then there exists a unique standard complex $C(a_1, \dots, a_{2n})$ such that $C$ splits as a direct sum $C \cong C(a_1, \dots, a_{2n}) \oplus A$ for some chain complex $A$ over $\mathcal{R}_1$.
\end{theorem}
Note that the exact phrasing of the theorem in \cite{dai2021more} is slightly different. In particular, the authors only claim $C \simeq C(a_1, \dots, a_{2n}) \oplus A$, but they actually prove the above stronger version.
\begin{remark}
We emphasize that knot-like complexes are chain complexes over $\mathcal{R}_1$. The splitting in Theorem \ref{classification} only works over $\mathcal{R}_1$ and, indeed, is the main reason for working over this smaller ring as opposed to $\mathcal{R}_\infty=\F[U,V]$.
\end{remark}

\section{Algebraic realizability}
In this section, we introduce the notion of algebraic realizability and classify algebraically realizable local equivalence classes of knot Floer complexes.
%Knot-like complexes are a certain formal generalization of the class of knot Floer complexes $\{ \CFK_{\mathcal{R}_1}(K) \ | \ K \subset S^3 \text{ a knot}\}$. Additional flexibility and  a purely algebraic character are utilized to classify them in Theorem \ref{classification}. However, when all is done, the question of which of these complexes actually arise from knots in $S^3$ remains.

\vspace{1em}

There are two closely related notions of realizability one can explore. Given a knot-like complex $C$, can it be realized as $\CFK_{\mathcal{R}_1}(K)$ up to local equivalence? What about up to chain homotopy equivalence? Both of these questions are difficult to answer. Some restrictions on the chain homotopy type can be found in \cite{hedden2018geography} and \cite{krcatovich2018restriction}. In this paper, we explore the restrictions on local equivalence classes. Following the list of algebraic properties of knot Floer complexes $\CFK_{\F[U,V]}(K)$ in Section \ref{subsection: Knot Floer complexes}, we give the following definition.

\begin{definition}\label{def:partially and fully realizable}
A free $\mathcal{R}_1$-module with an endomorphism and a distinguished basis $(C, \partial_C, B)$ is
\begin{itemize}[label = --]
    \item \emph{partially realizable} if it is a mod $UV$ reduction of a free chain complex over $\mathcal{R}_2$ with a distinguished basis $(D, \partial_D, B)$,
    \item \emph{fully realizable} if it is a mod $UV$ reduction of a free chain complex over $\mathcal{R}_\infty$ with a distinguished basis $(D, \partial_D, B)$.
\end{itemize}
A chain complex $D$ over $\mathcal{R}_2$ or $\mathcal{R}_\infty$ whose mod $UV$ reduction is $C$ is called a \emph{partial realization} or a \emph{full realization} of $C$.
\end{definition}
\begin{definition}\label{def:algebraically realizable}
A (symmetric) knot-like complex $C$ is \emph{algebraically realizable} if it is fully realizable and its full realization (is symmetric and) has the correct homology.

A local equivalence class of knot-like complexes is \emph{algebraically realizable} if one of its representatives is algebraically realizable.
\end{definition}

\begin{example}
The local equivalence class of $C(1, -1, 3, -2)$ is algebraically realizable because $C(1, -1, 3, -2)$ lifts to a $\Z\oplus\Z$ graded chain complex over $\F[U,V]$ with the correct homology. See Figure \ref{fig:C(1,-1,3,-2)}.
\begin{figure}[t]
    \centering
    \begin{tikzpicture}[scale=0.9]
    \draw[step=1.0,gray,thin] (0.5,0.5) grid (5.5,4.5);
    \draw[black, very thick] (1, 4) -- (2, 4) -- (2, 3) -- (5, 3) -- (5, 1);
    \filldraw[black] (1,4) circle (2pt) node[anchor=south east]{$x_0$};
    \filldraw[black] (2,4) circle (2pt) node[anchor=south east]{$x_1$};
    \filldraw[black] (2,3) circle (2pt) node[anchor=south east]{$x_2$};
    \filldraw[black] (5,3) circle (2pt) node[anchor=south east]{$x_3$};
    \filldraw[black] (5,1) circle (2pt) node[anchor=south east]{$x_4$};
    \end{tikzpicture}
    \caption{The complex $C(1,-1,3,-2)$ and its local equivalence class are algebraically realizable.}
    \label{fig:C(1,-1,3,-2)}
\end{figure}
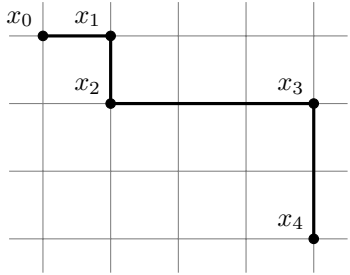

More generally, let $a_1, \dots, a_{2n}$ be any sequence of nonzero integers with alternating signs. Then the complex $C(a_1, \dots, a_{2n})$ lifts to a $\Z\oplus\Z$ graded chain complex over $\F[U,V]$ with the correct homology. It follows that the local equivalence class of $C(a_1, \dots, a_{2n})$ is algebraically realizable.
\end{example}

\begin{example}
The local equivalence class of $C(2,2)$ is algebraically realizable. To see this, observe that Figure \ref{fig:C(2,2)} depicts a bigraded chain complex over $\F[U,V]$ with the correct homology whose mod $UV$ reduction is locally equivalent to $C(2,2)$.
\begin{figure}[t]
    \centering
    \begin{tikzpicture}[scale=0.9]
    \draw[step=1.0,gray,thin] (0.5,0.5) grid (6.5,5.5);
    \draw[black, very thick] (1, 1) -- (4, 1) -- (4, 3) -- (6, 3) -- (6, 5) -- (2,5) -- (2,2) -- (5,4);
    \draw[red, very thick] (5,4) -- (5, 2) -- (3,2);
    \draw[black, very thick] (3,2) -- (1,1);
    \draw[black, very thick] (3,2) -- (6,5);
    \draw[black, very thick] (4,1) -- (6,3);
    \filldraw[black] (1,1) circle (2pt) node[anchor=south east]{$z$};
    \filldraw[black] (4,1) circle (2pt);
    \filldraw[black] (4,3) circle (2pt);
    \filldraw[black] (6,3) circle (2pt);
    \filldraw[black] (6,5) circle (2pt);
    \filldraw[black] (2,5) circle (2pt);
    \filldraw[black] (2,2) circle (2pt) node[anchor=south east]{$z$};
    \filldraw[black] (5,4) circle (2pt) node[anchor=north west]{$x_2$};
    \filldraw[black] (5,2) circle (2pt) node[anchor=north west]{$x_1$};
    \filldraw[black] (3,2) circle (2pt) node[anchor=north west]{$x_0$};
    \end{tikzpicture}
    \caption{The local equivalence class of the complex $C(2,2)$ is algebraically realizable. The $C(2,2)$ summand over the ring $\frac{\F[U,V]}{UV}$ is drawn in red.}
    \label{fig:C(2,2)}
\end{figure}
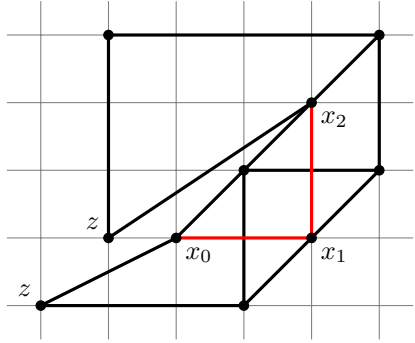
\end{example}
\begin{example}
The local equivalence class of $C(1, n)$ is not algebraically realizable for any $n \geq 1$. To see this, note that by Theorem \ref{classification}, any representative of this class is isomorphic over $\mathcal{R}_1$ to $C(1, n)\oplus A$ for some chain complex $A$. Consider a hypothetical algebraic realization and delete its diagonal arrows. This splits the complex into two components, one of which is $C(1,n)$. However, it is impossible to restore $\partial^2=0$ over $\F[U,V]$ in $C(1, n)$ by only adding some diagonal arrows, regardless of the presence of $A$. This is a contradiction, so the local equivalence class of $C(1,n)$ is not algebraically realizable.
\end{example}
We now give a more general example of an obstruction to algebraic realizability of certain standard complexes and their local equivalence classes.
\begin{example}
Let $C$ be an algebraically realizable standard complex. If the type of $C$ contains the substring $a, 1, b$ with $b > 0$, then $-b < a < 0$. This is because we must have $\partial^2=0$ over $\F[U,V]$, which necessitates that there be a dashed diagonal arrow as shown below.

$$
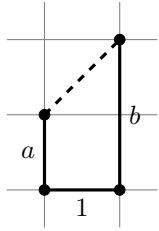
\begin{tikzpicture}
    \draw[step=1.0,gray,thin] (0.5,0.5) grid (2.5,3.5);
    \draw[black, very thick] (1,2) -- (1,1) -- (2,1) -- (2,3);
    \draw[black, very thick, dashed] (1,2) -- (2,3);
    \node at (1,1.5)[anchor = east]{$a$};
    \node at (2,2)[anchor = west]{$b$};
    \node at (1.5,1)[anchor = north]{$1$};
    \filldraw[black] (1,2) circle (2pt);
    \filldraw[black] (1,1) circle (2pt);
    \filldraw[black] (2,1) circle (2pt);
    \filldraw[black] (2,3) circle (2pt);
\end{tikzpicture}
$$

This example has several symmetry properties. In particular, if there is a substring $a, 1, b$ with $a > 0$, then $-a < b < 0$. There exist similar restrictions on the substrings of the form $c, -1, d$. Moreover, it is easy to see that if a substring of this form obstructs algebraic realizability of $C$, it actually obstructs algebraic realizability of an entire local equivalence class. For example, the local equivalence classes of $C(2, 1, -3, 1)$ and $C(-8, 2, 1, 2)$ are not algebraically realizable.
\end{example}

We show that phenomena of this type are the only potential obstructions to algebraic realizability of local equivalence classes of knot-like complexes. In precise language, we prove the following proposition.

\begin{proposition}\label{proposition}
The following are equivalent.
\begin{enumerate}
    \item The local equivalence class of $C(a_1, \dots, a_{2n})$ is algebraically realizable.
    \item $C(a_1, \dots, a_{2n})$ is partially realizable.
\item $C(N_1 \ | \ a_1, \dots, a_{2n} \ | -N_2)$ is partially realizable for all sufficiently large $N_1$ and $N_2$.
\end{enumerate}
\end{proposition}
\begin{proof}
    This is the content of Lemma \ref{lemma:local class real->partially real}, Lemma \ref{lemma:partially real->extended partially real}, and Lemma \ref{lemma:extended partially real->local class real}.
\end{proof}
Let us briefly remark on the notation we use in the proofs of the following lemmas. In triples of the form $(C_i, \partial_{C_i}, B_{C_i})$, the subscript $i$ indicates that $C_i$ is an $\mathcal{R}_i$-module. To guide the reader, we note that
\begin{itemize}
    \item [] $D_\infty \rightsquigarrow D_2 \rightsquigarrow E_2$,
    \item [] $E_2 \rightsquigarrow F_2$,
    \item [] $F_2 \rightsquigarrow F_\infty \rightsquigarrow G_\infty$
\end{itemize}
are the main modules appearing in the proofs of Lemma \ref{lemma:local class real->partially real}, Lemma \ref{lemma:partially real->extended partially real}, and Lemma \ref{lemma:extended partially real->local class real} respectively.

\begin{lemma}\label{lemma:local class real->partially real}
    If the local equivalence class of $C(a_1, \dots, a_{2n})$ is algebraically realizable, then $C(a_1, \dots, a_{2n})$ is partially realizable. 
\end{lemma}
\begin{proof}
Let $(D_\infty, \partial_{D_\infty}, B_{D_\infty})$ be an algebraic realization of the local equivalence class of $C(a_1, \dots, a_{2n})$. By Theorem \ref{classification}, after an appropriate change of basis, the mod $UV$ reduction of $D_\infty$ splits as a direct sum $C(a_1, \dots, a_{2n}) \oplus A$ for some free chain complex over $\mathcal{R}_1$ with a distinguished basis $(A, \partial_{A}, B_{A})$. In other words, $(D_\infty, \partial_{D_\infty}, \{x_0, \dots, x_{2n}\}\cup B_{A})$ is a full realization of $C(a_1, \dots, a_{2n})\oplus A$.

Let $(D_2, \partial_{D_2}, \{x_0, \dots, x_{2n}\}\cup B_{A})$ be the mod $U^2V^2$ reduction of $D_\infty$. Note that $D_2$ is a chain complex over $\mathcal{R}_2$ since $D_\infty$ is a chain complex over $\mathcal{R}_\infty$. Let now $E_2=\mathcal{R}_2\langle x_0, \dots, x_{2n}\rangle$ be the submodule of $D_2$ generated by $\{x_0, \dots, x_{2n}\}$ and let $\partial_{E_2}: E_2 \to E_2$ be defined via $\langle \partial_{E_2}x_i, U^aV^b x_j \rangle = \langle \partial_{D_2}x_i, U^aV^b x_j \rangle$ for all $x_i, x_j$ and extended to an $\mathcal{R}_2$-module homomorphism. In other words, $\partial_{E_2}$ is the part of the differential $\partial_{D_2}$ that is `happening in $E_2$'.

We claim that $(E_2, \partial_{E_2}, \{x_0, \dots, x_{2n}\})$ is a partial realization of $C(a_1, \dots, a_{2n})$. Assume for the contradiction that this is not the case. Since $C(a_1, \dots, a_{2n})$ is a mod $UV$ reduction of $E_2$, it must be the case that $(E_2, \partial_{E_2})$ is not a chain complex. So there exist $x_i, x_j$ and $a, b \in \N$ with $\min (a,b) \leq 1$ such that $\langle \partial^2_{E_2}x_i, U^aV^bx_j \rangle =1$. But $\partial^2_{D_2}x_i =0$ so there is a composition of two arrows in $D_2$, one from $x_i$ to $y$, followed by one from $y$ to $x_j$ for some $y \in B_A$. But $\min (a,b) \leq 1$ so both arrows cannot be diagonal. This means that one of the arrows is vertical or horizontal, which implies that $y \in \{x_0, \dots, x_{2n}\}$. This is a contradiction so $E_2$ is a chain complex and a partial realization of $C(a_1, \dots, a_{2n})$.
\end{proof}

\begin{lemma}\label{lemma:partially real->extended partially real}
If $C(a_1, \dots, a_{2n})$ is partially realizable, then $C(N_1 \ | \ a_1, \dots, a_{2n} \ | -N_2)$ is partially realizable for all sufficiently large $N_1, N_2 \in\N$.
\end{lemma}
\begin{proof}
We will show that taking any $N_1 \geq \max \{ |a_1|, \dots, |a_{2n}| \}+1$ works. Assume $C(a_1, \dots, a_{2n})$ is partially realizable and let $(E_2, \partial_{E_2}, B_{E_2})$ be its partial realization. Consider the free $\mathcal{R}_2$-module $F_2$ with a distinguished basis $B_{F_2}=B \cup \{ x_{-1} \}$ and whose endomorphism is $\partial_{E_2}$ together with a vertical arrow of length $N_1$ from $x_0$ to $x_{-1}$. Consider the set $X = \{ x\in B_{F_2} \ |  \ \partial_{F_2}^2 x \neq 0 \}$ of elements witnessing that $F_2$ is not a chain complex. Note that since $E_2$ is a chain complex, all elements of $X$ have to come from the extra vertical arrow we just added. In other words, there is an alternative description of $X$ as $\{ x\in B_{F_2} \ |  \ \partial_{F_2} x = UV^bx_0 \text{ for some }b\geq 0\}$ using that $N_1 \geq 2$.

Since the mod $UV$ reduction of $E_2$ is the standard complex $C(a_1, \dots, a_{2n})$, every $x \in B_{F_2}$ is adjacent to exactly one vertical arrow in $F_2$. In particular, this is true for every $x \in X$. To prove that all vertical arrows adjacent to elements of $X$ are outgoing, assume for the contradiction that there is an $x\in X$ with an incoming vertical arrow from $y$ to $x$. Then passing through $x$ is the only way of going from $y$ to $x_0$ in $E_2$ in two steps. Any other way would have to first go diagonally (because $y$ is adjacent to a unique vertical arrow that has already been used) and then vertically to $x_0$, but $x_0$ has no incoming vertical arrow. Therefore $\langle\partial_{E_2}^2y, UV^bx_0 \rangle = 1$ for some $b$. This is a contradiction with the fact that $E_2$ is a chain complex. Therefore, we have established that all $x \in X$ have outgoing vertical arrows. Let $Y$ be the set of their other endpoints. By our choice of $N_1$, we can draw diagonal arrows in $F_2$ from every element $y\in Y$ to $x_{-1}$. Therefore $\partial_{F_2}^2 x = 0$ in $F_2$ for all $x \in X$.

We claim that $F_2$ is a free chain complex over $\mathcal{R}_2$ with a distinguished basis $B_{F_2}$. Indeed, the only thing that could go wrong is that $\partial_{F_2}^2 z \neq 0$ for some $z \notin X$ that has an arrow to an element of $Y$. But such arrows cannot be vertical (because the unique vertical arrow at every element of $Y$ goes to an element of $X$). Any non-vertical arrow from $z$ to $y \in Y$ followed by a diagonal arrow from $y$ to $x_{-1}$ vanishes over $\mathcal{R}_2$, so indeed $\partial_{F_2}^2=0$ as required. A completely analogous construction works at the other endpoint to construct a partial realization of $C(N_1 \ | \ a_1, \dots, a_{2n} \ | -N_2)$.
\end{proof}
\begin{lemma}\label{lemma:extended partially real->local class real}
If $C(N_1 \ | \ a_1, \dots, a_{2n} \ | -N_2)$ is partially realizable for all sufficiently large $N_1, N_2 \in\N$, then the local equivalence class of $C(a_1, \dots, a_{2n})$ is algebraically realizable.
\end{lemma}
\begin{proof}
    Our proof proceeds through an intermediate step. We prove that there exists a chain complex $A$ over $\mathcal{R}_1$ such that $C(N_1 \ | \ a_1, \dots, a_{2n} \ | -N_2) \oplus A$ is fully realizable. Then we adapt this construction to produce a suitable chain complex over $\F[U,V]$ algebraically realizing the local equivalence class of $C(a_1, \dots, a_{2n})$.

    \vspace{1em}

    Let $(F_2, \partial_{F_2}, B_{F_2})$ be a partial realization of $C(N_1 \ | \ a_1, \dots, a_{2n} \ | -N_2)$ with a basis $B_{F_2} = \{x_{-1}, x_0, \dots, x_{2n}, x_{2n+1} \}$. Define $B_{F_\infty} = \{ x_{-1}, \dots, x_{2n+1}, y_{-1}, \dots, y_{2n+1}\}$ and let $F_\infty$ be the free $\mathcal{R}_\infty$-module generated by $B_{F_\infty}$. We equip $F_\infty$ with an endomorphism $\partial_{F_\infty}$ and a $\Z \oplus \Z$ grading as in the examples in Figure \ref{fig:Case C(-,...,+)} and Figure \ref{fig:Case C(-,...,+) asymmetric}.
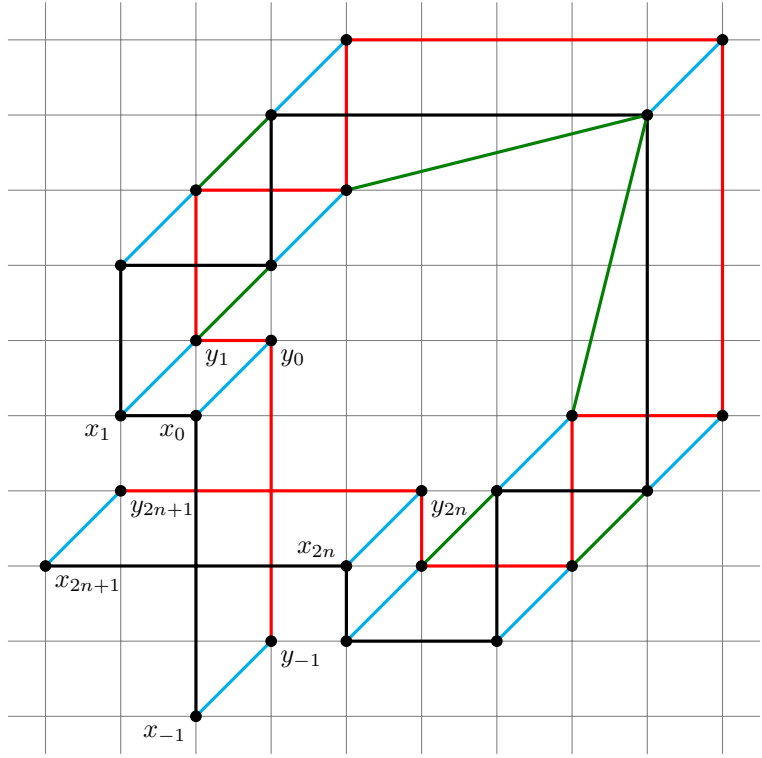
\begin{figure}[t]
    \centering
    \begin{tikzpicture}[scale=1.0]
    \draw[step=1.0,gray,thin] (-0.5,-0.5) grid (9.5,9.5);
    \draw[red, very thick] (1,3) -- (5, 3) -- (5, 2) -- (7, 2) -- (7, 4) -- (9, 4) -- (9,9) -- (4,9) -- (4,7) -- (2,7) -- (2,5) -- (3,5) -- (3,1);
    \draw[black, very thick] (0,2) -- (4, 2) -- (4, 1) -- (6, 1) -- (6, 3) -- (8, 3) -- (8,8) -- (3,8) -- (3,6) -- (1,6) -- (1,4) -- (2,4) -- (2,0);

    \draw[cyan, very thick] (2,0) -- (3,1);
    \draw[cyan, very thick] (0,2) -- (1,3);
    
    \draw[cyan, very thick] (6,1) -- (7,2);
    \draw[dark green, very thick] (7,2) -- (8,3);
    \draw[cyan, very thick] (8,3) -- (9,4);

    \draw[cyan, very thick] (1,6) -- (2,7);
    \draw[dark green, very thick] (2,7) -- (3,8);
    \draw[cyan, very thick] (3,8) -- (4,9);
    
    \draw[cyan, very thick] (4,1) -- (5,2);
    \draw[dark green, very thick] (5,2) -- (6,3);
    \draw[cyan, very thick] (6,3) -- (7,4);
    \draw[dark green, very thick] (7,4) -- (8,8);

    \draw[cyan, very thick] (1,4) -- (2,5);
    \draw[dark green, very thick] (2,5) -- (3,6);
    \draw[cyan, very thick] (3,6) -- (4,7);
    \draw[dark green, very thick] (4,7) -- (8,8);
    
    \draw[cyan, very thick] (4,2) -- (5,3);
    \draw[cyan, very thick] (2,4) -- (3,5);
    \draw[cyan, very thick] (8,8) -- (9,9);
    
    \filldraw[black] (6,3) circle (2pt);
    \filldraw[black] (8,3) circle (2pt);
    \filldraw[black] (8,8) circle (2pt);
    \filldraw[black] (3,8) circle (2pt);
    \filldraw[black] (3,6) circle (2pt);
    \filldraw[black] (1,6) circle (2pt);
    \filldraw[black] (1,4) circle (2pt) node[anchor=north east]{$x_1$};
    \filldraw[black] (2,4) circle (2pt) node[anchor=north east]{$x_0$};

    \filldraw[black] (5,3) circle (2pt) node[anchor=north west]{$y_{2n}$};
    \filldraw[black] (5,2) circle (2pt) node[anchor=north west]{};
    \filldraw[black] (7,2) circle (2pt);
    \filldraw[black] (7,4) circle (2pt); 
    \filldraw[black] (9,4) circle (2pt);
    \filldraw[black] (9,9) circle (2pt);
    \filldraw[black] (4,9) circle (2pt);
    \filldraw[black] (4,7) circle (2pt);
    \filldraw[black] (2,7) circle (2pt);
    \filldraw[black] (2,5) circle (2pt) node[anchor=north west]{$y_1$};
    \filldraw[black] (3,5) circle (2pt) node[anchor=north west]{$y_0$};
    \filldraw[black] (4,2) circle (2pt) node[anchor=south east]{$x_{2n}$};
    \filldraw[black] (4,1) circle (2pt) node[anchor=north west]{};
    \filldraw[black] (6,1) circle (2pt);

    \filldraw[black] (2,0) circle (2pt) node[anchor=north east]{$x_{-1}$};
    \filldraw[black] (0,2) circle (2pt) node[anchor=north west]{$x_{2n+1}$};
    \filldraw[black] (1,3) circle (2pt) node[anchor=north west]{$y_{2n+1}$};
    \filldraw[black] (3,1) circle (2pt) node[anchor=north west]{$y_{-1}$};
    \end{tikzpicture}
    \caption{A partial realization $F_2$ of the starting complex $C(N_1 \ | \ a_0, \dots, a_{2n}, \ | -N_2)$ is drawn in black. The depicted structure $F_\infty$ is a chain complex over $\F[U,V]$.}
    \label{fig:Case C(-,...,+)}
\end{figure}

Let us describe the construction in words: $F_\infty$ contains two copies of $F_2$, one generated by the $x_i$ (drawn in black) and one generated by the $y_i$ (drawn in red). The red copy of $F_2$ is on the same diagonal one step above the black copy of $F_2$. We draw the diagonal arrows from $y_i$ to $x_i$ for all $i$ so as to make $\langle \partial_{F_\infty} y_i, UVx_i \rangle = 1$. These are the light blue arrows. Finally, if $a, b \in \N$ are such that $\langle \partial_{F_2}^2 x_{i}, U^{a}V^{b}x_j \rangle = 1$ (note that this implies $a, b \geq 2$ since $F_2$ is a chain complex over $\mathcal{R}_2$), we draw a diagonal arrow from $x_{i}$ to $y_{j}$. These are the green arrows and they make $\langle \partial_{F_\infty} x_{i}, U^{a-1}V^{b-1}y_j \rangle = 1$.

\vspace{1em}

The $\Z \oplus \Z$ gradings of the basis elements $x_i$ of $F_\infty$ are the same as their gradings in $F_2$. The gradings of the basis elements $y_i$ are uniquely determined by the requirement that $\partial$ have degree $(-1, -1)$. Explicitly, we have $\gr(y_i)=\gr(x_i)-(1, 1)$.

\vspace{1em}

%\begin{itemize}
%    \item The bigradings of $x_i$ are the same as in $C$.
%    \item The bigradings of $y_i$ are $\gr(y_i) = \gr(x_i)-(1,1)$.
%    \item The bigradings of $z_i$ are a little bit more delicate and will be specified later as the unique bigradings so that the map $\partial$ we define has bidegree $(-1,-1)$. For now we draw the generator $z_3$ in the plane in the position one lower than the row of $x_{2n}$ and one to the left of the column of $x_0$. We further draw $z_1$ in the column of $y_0$ and the row of $z_3$ and we draw $z_2$ in the column of the second copy of $z_3$ and the row of $y_{2n}$.
%\end{itemize}

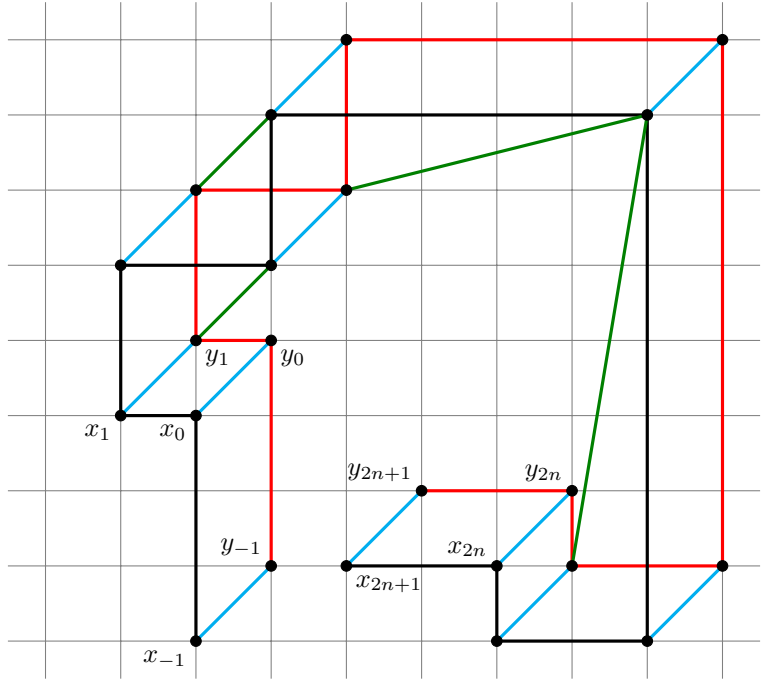
\begin{figure}[t]
    \centering
    \begin{tikzpicture}[scale=1.0]
    \draw[step=1.0,gray,thin] (-0.5,1.5) grid (9.5,10.5);
    \draw[red, very thick] (5,4) -- (7,4) -- (7,3) -- (9, 3) -- (9,10) -- (4,10) -- (4,8) -- (2,8) -- (2,6) -- (3,6) -- (3,3);
    \draw[black, very thick] (4, 3) -- (6, 3) -- (6, 2) -- (8, 2) -- (8,9) -- (3,9) -- (3,7) -- (1,7) -- (1,5) -- (2,5) -- (2,2);

    \draw[cyan, very thick] (4,3) -- (5,4);
    \draw[cyan, very thick] (2,2) -- (3,3);
    
    \draw[cyan, very thick] (6,2) -- (7,3);
    \draw[cyan, very thick] (8,2) -- (9,3);
    \draw[cyan, very thick] (6,3) -- (7,4);

    \draw[cyan, very thick] (1,7) -- (2,8);
    \draw[dark green, very thick] (2,8) -- (3,9);
    \draw[cyan, very thick] (3,9) -- (4,10);
    
    \draw[dark green, very thick] (7,3) -- (8,9);

    \draw[cyan, very thick] (1,5) -- (2,6);
    \draw[dark green, very thick] (2,6) -- (3,7);
    \draw[cyan, very thick] (3,7) -- (4,8);
    \draw[dark green, very thick] (4,8) -- (8,9);
    
    \draw[cyan, very thick] (2,5) -- (3,6);
    \draw[cyan, very thick] (8,9) -- (9,10);

    \filldraw[black] (8,9) circle (2pt);
    \filldraw[black] (3,9) circle (2pt);
    \filldraw[black] (3,7) circle (2pt);
    \filldraw[black] (1,7) circle (2pt);
    \filldraw[black] (1,5) circle (2pt) node[anchor=north east]{$x_1$};
    \filldraw[black] (2,5) circle (2pt) node[anchor=north east]{$x_0$};

    \filldraw[black] (7,4) circle (2pt) node[anchor=south east]{$y_{2n}$};
    \filldraw[black] (7,3) circle (2pt) node[anchor=north west]{};
    \filldraw[black] (8,2) circle (2pt); 
    \filldraw[black] (9,3) circle (2pt);
    \filldraw[black] (9,10) circle (2pt);
    \filldraw[black] (4,10) circle (2pt);
    \filldraw[black] (4,8) circle (2pt);
    \filldraw[black] (2,8) circle (2pt);
    \filldraw[black] (2,6) circle (2pt) node[anchor=north west]{$y_1$};
    \filldraw[black] (3,6) circle (2pt) node[anchor=north west]{$y_0$};
    \filldraw[black] (6,3) circle (2pt) node[anchor=south east]{$x_{2n}$};
    \filldraw[black] (6,2) circle (2pt) node[anchor=north west]{};

    \filldraw[black] (4,3) circle (2pt) node[anchor=north west]{$x_{2n+1}$};
    \filldraw[black] (2,2) circle (2pt) node[anchor=north east]{$x_{-1}$};
    \filldraw[black] (5,4) circle (2pt) node[anchor=south east]{$y_{2n+1}$};
    \filldraw[black] (3,3) circle (2pt) node[anchor=south east]{$y_{-1}$};
    \end{tikzpicture}
    \caption{A partial realization $F_2$ of the starting complex is drawn in black. The depicted structure $F_\infty$ is a chain complex over $\F[U,V]$.}
    \label{fig:Case C(-,...,+) asymmetric}
\end{figure}

We shall now check that $F_2$ is in fact a chain complex, \emph{i.e.}, that $\partial_{F_\infty}^2=0$. Since we are working over $\F=\Z/2$, it is enough to verify that, starting at any distinguished basis element, there is an even number of ways to reach any other generator in two steps. It is useful to have the following relationships in mind:
\begin{itemize}[label=--]
    \item black $\circ$ black $=$ black $\circ$ black \hfill (some black pairs)
    \item \textcolor{cyan}{blue} $\circ$ \textcolor{dark green}{green} $= \ $ \textcolor{black}{black} $\circ$ \textcolor{black}{black} \hfill (the remaining black pairs)
    \item \textcolor{dark green}{green} $\circ$ \textcolor{black}{black} $= \ $\textcolor{red}{red} $\circ$ \textcolor{dark green}{green}
    \item \textcolor{red}{red} $\circ$ \textcolor{red}{red} $=$ \textcolor{red}{red} $\circ$ \textcolor{red}{red} \hfill (some red pairs)
    \item  \textcolor{dark green}{green} $\circ$ \textcolor{cyan}{blue} $=$ \textcolor{red}{red} $\circ$ \textcolor{red}{red} \hfill (the remaining red pairs)
     \item \textcolor{cyan}{blue} $\circ$ \textcolor{red}{red} $= \ $\textcolor{black}{black} $\circ$ \textcolor{cyan}{blue}.
\end{itemize}
The exact meaning of these equations is now formalized in the careful and detailed write-up of the proof that $\partial_{F_\infty}^2=0$.

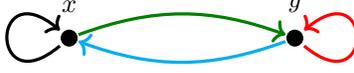
\begin{figure}[t]
    \centering
    \begin{tikzpicture}[scale=1.5]
        \def\a{0.15}
        \draw (-1,\a) node[anchor = south]{$x$};
        \draw (1,\a) node[anchor = south]{$y$};

        \node (X) at (-1,0){};
        \node (Y) at (1,0){};
        
        \filldraw[black] (X) circle (2pt);
        \filldraw[black] (Y) circle (2pt);
        \draw[->, dark green, very thick] (X) to[out=20, in=160] (Y);
        \draw[->, cyan, very thick] (Y) to[out=-160, in=-20] (X);
        \draw[->, black, very thick] (X) to[out=-140, in=140, looseness = 15] (X);
        \draw[->, red, very thick] (Y) to[out=-40, in=40, looseness = 15] (Y);
    \end{tikzpicture}
    \caption{A schematic depiction of sources and targets of arrows of each color.}
    \label{fig:Arrow analysis}
\end{figure}

\begin{itemize}
    \item Let's pair all ways of travelling from $x_i$ to $U^aV^bx_j$ in two steps. 
    %All arrows originating at $x_i$ where $i \notin \{0, 2n\}$ are either black or green. 
    %There are four possible colour sequences starting with one of these two colors and they can be paired as follows.
    \begin{itemize}
        \item If $\min (a, b) \leq 1$, then $\langle \partial_{F_2}^2 x_i, U^aV^bx_j \rangle = 0$ since $\partial_{F_2}$ is a differential over $\mathcal{R}_2$. There are no extra ways to travel between $x_i$ and $x_j$ in $F_\infty$ since there is an outgoing green arrow from $x_i$ only if $\langle \partial_{F_2}^2 x_i, U^aV^bx_j \rangle = 1$.
        %In all other cases, assume that at least one of the black arrows is not in $C$.
        \item Otherwise $a, b \geq 2$. Assume $\langle \partial_{F_2}^2 x_i, U^aV^bx_j \rangle = 1$. By construction, there is a green arrow in $F_\infty$ from $x_i$ to $y_j$, which, together with the blue arrow from $y_{j}$ to $x_{j}$ cancels the two black arrows. The converse is true as well. The existence of a green arrow starting at $x_i$ implies $\langle \partial_{F_2}^2 x_i, U^aV^bx_j \rangle = 1$, \emph{i.e.} the existence of two consecutive black arrows with a non-zero composition.
        %\item \textcolor{black}{black} $\circ$ \textcolor{black}{black} $= \ $\textcolor{cyan}{ blue} $\circ$ \textcolor{dark green}{green}.
    \end{itemize}
    \item Let's pair all ways of travelling from $x_i$ to $U^aV^by_j$ in two steps. Let there be a black arrow from $x_i$ to $x_{k}$ followed by a green arrow to $y_j$. The green arrow from $x_{k}$ to $y_{j}$ implies that $\langle \partial_{F_2}^2 x_k, U^{a'}V^{b'}x_j \rangle = 1$ for some $a'$, $b'$. In other words, there is a sequence of two black arrows, one from $x_k$ to $x_l$, followed by one from $x_l$ to $x_j$. Therefore, there is a green arrow from $x_i$ to $y_l$. Together with the red arrow from $y_l$ to $y_j$, this cancels the original two black arrows.
        
    The converse is true as well. Any green arrow followed by a red arrow is matched with a sequence of two black arrows by a similar reasoning.
    %\item Any way of travelling from $x_i$ to $U^aV^bz_j$ in two steps will necessarily involve $x_0$ or $x_{2n}$. Since $a_1 <0$ and $a_{2n}>0$, it suffices to check that $\partial^2 x_0 = 0$ and $\partial^2 x_{2n}=0$. Any sequence of arrows starting with a black or a green arrow is cancelled as in the above case. The generators $x_0$ and $x_{2n}$ additionally each have a single outgoing orange arrow, the one into $z_3$, but there are no further arrows out of $z_3$.

    \item Let's pair all ways of travelling from $y_i$ to $U^aV^by_j$ in two steps.
    \begin{itemize}
        \item If $\min (a, b) \leq 1$, then $\langle \partial_{F_2}^2 x_i, U^aV^bx_j \rangle = 0$ by assumption and since the $y_i$'s are just a copy of $F_2$, this pairs some red arrows together.
        \item Otherwise $a, b \geq 2$. Assume $\langle \partial_{F_2}^2 x_i, U^aV^bx_j \rangle = 1$. By construction, there is a green arrow in $F_\infty$ from $x_i$ to $y_j$, which, together with the blue arrow from $y_{i}$ to $x_{i}$ cancels the two red arrows. The converse is true as well. The existence of a green arrow starting at $x_i$ implies $\langle \partial_{F_2}^2 x_i, U^aV^bx_j \rangle = 1$, \emph{i.e.} we get the existence of two consecutive black arrows from $x_i$ to $x_j$ with a non-zero composition. Therefore also the existence of two consecutive red arrows from $y_i$ to $y_j$ with a non-zero composition.
    \end{itemize}
    \item Finally, let's now pair all ways of travelling from $y_i$ to $U^aV^bx_j$ in two steps. There is just one pair of ways how this can be achieved. The composition of a red arrow from $y_i$ to $y_j$ with the blue arrow from $y_j$ to $x_j$ is cancelled by the composition of a blue arrow from $y_i$ to $x_i$ with the black arrow from $x_i$ to $x_j$.
    %\item Any way of travelling from $y_i$ to $U^aV^bz_j$ in two steps will necessarily involve $y_0$ or $y_{2n}$. Since $a_1 <0$ and $a_{2n}>0$, it suffices to check that $\partial^2 y_0 = 0$ and $\partial^2 y_{2n}=0$. Any sequence of arrows starting with a red or a blue arrow is cancelled as in the above case. Additionally, there is a non-zero composition of orange arrows from $y_0$ to $z_1$ and from $z_1$ to $z_3$. This composition is cancelled by the composition of the blue arrow from $y_0$ to $x_0$ and an orange arrow from $x_0$ to $z_3$. Similarly, the orange arrows from $y_{2n}$ to $z_2$ and $z_2$ to $z_3$ are cancelled by the blue arrow from $y_{2n}$ to $x_{2n}$ and an orange arrow from $x_{2n}$ to $z_3$.
    %\item Finally, $\partial^2 z_i = 0$ for all $i$ since $z_3$ has no outgoing arrows.
\end{itemize}
We have now verified that $F_\infty$ is a $\Z\oplus\Z$ graded free chain complex over $\mathcal{R}_\infty$. Its mod $UV$ reduction is isomorphic to a direct sum of two copies of $C(N, a_1, \dots, a_{2n}, -N)$, one black and one red. Therefore $F_\infty$ is not an algebraic relalization, since its homology is not correct yet. We achieve this by modifying the construction as follows.

\vspace{1em}

It was proven in Lemma \ref{lemma:partially real->extended partially real} that the exact lengths $N_1, N_2$ of the extensions are not important. In particular, the extensions can always be elongated. This is important for the following construction.

Let $s=\frac{1}{2}\sum_{i=1}^{2n}\mathrm{sgn}(a_i) \in \Z$. Consider the free $\mathcal{R}_\infty$-module $G_\infty$ generated by $\{x_0, \dots, x_{2n}, y_{-1}, \dots, y_{2n+1}, z\}$. The elements $x_i$ and $y_i$ for $i\in\{0, \dots, 2n\}$ are drawn in the plane as earlier and the basis element $z$ is drawn in the plane far down and to the left of any other generator. If $s=0$, the basis element $y_{-1}$ is drawn in the row of $z$ and the column of $y_0$ and the basis element $y_{2n-1}$ is drawn in the row of $y_{2n}$ and the column of $z$. If $s\neq 0$, we draw the second copy of $z$ $|s|$ places to the left and below of the first copy. If $s > 0$, the basis element $y_{-1}$ is drawn in the row of the second copy of $z$ and the column of $y_{-1}$ and the basis element $y_{2n+1}$ is drawn in the row of $y_{2n}$ and the column of the first copy of $z$. If $s < 0$, the basis element $y_{-1}$ is drawn in the row of the first copy of $z$ and the column of $y_0$ and the basis element $y_{2n+1}$ is drawn in the row of $y_{2n}$ and the column of the second copy of $z$.

Define an endomorphism of $G_\infty$ by with the same rules as in $F_\infty$, with the convention that $z$ assumes the roles of both $x_{-1}$ and $x_{2n+1}$. We can always do that since $z$, or both of its copies, lies strictly below and to the left of both $x_{-1}$ and $x_{2n+1}$.

Note that $G_\infty$ is a free chain complex over $\mathcal{R}_\infty$ for the same reason that $F_\infty$ is. To be explicit: we have slightly changed the $\Z\oplus\Z$ gradings of the generators and we have identified $x_{-1}$ with $x_{2n+1}$, but the differential $\partial_{G_\infty}$ connects the same pairs of generators as $\partial_{F_\infty}$, so $\partial_{G_\infty}^2=0$. After setting $U=0$ and ignoring $V$-torsion, the homology of $G_\infty$ is generated by $x_0$ and after setting $V=0$ and ignoring $U$-torsion, the homology of $G_\infty$ is generated by $x_{2n}$. Finally, our construction is symmetric: if $F_2$ is symmetric, then $F_\infty$ and $G_\infty$ are symmetric as well. Therefore, the local equivalence class of $C(a_1, \dots, a_{2n})$ is algebraically realized by the complex $G_\infty$.
 \end{proof}
\begin{corollary}\label{corollary}
    The local equivalence class of $C(a_1, \dots, a_{2n})$ with $|a_i|\geq 2$ for all $i$ is algebraically realizable.
\end{corollary}
\begin{proof}
    The standard complex $C(a_1, \dots, a_{2n})$ is partially realizable provided $|a_i|\geq 2$ for all $i$, so this follows immediately from Proposition \ref{proposition}.
\end{proof}
Proposition \ref{proposition} reduces the question of algebraic realizability of local equivalence classes to an easier question of partial realizability of actual standard or extended standard complexes. All that is left is to provide an algorithm that solves the latter question and this is what we do now.

\begin{algorithm}\label{algorithm}
Let $C(a_1, \dots, a_{2n})$ be a standard complex considered as a free $\mathcal{R}_2$-module with an endomorphism and a distinguished basis $(C, \partial, B)$ where $B = \{x_{0}, \dots, x_{2n}\}$. We first describe the algorithm that determines when $C$ is partially realizable and later prove its correctness.

We begin with the $\mathcal{R}_2$-module $C$. At each step, if we encounter a particular situation, we try to augment the differential slightly by adding a diagonal arrow between two of its generators. More precisely; if there exist $x_i$, $x_j \in B$ such that $\langle \partial^2 x_i, U^aV^b x_j\rangle = 1$ where $1 \in \{a, b\}$, then at least one of the two arrows connecting them must be non-diagonal (if they are both diagonal, then $a, b \geq 2$). We distinguish between four cases:
\begin{enumerate}
    \item $b=1$: In this case, one of the two arrows must be horizontal.
    \begin{enumerate}
        \item The first arrow is horizontal. Depending on the parity of $i$, the horizontal arrow goes from $x_i$ to $x_{i\pm 1}$ and then there is a diagonal or a vertical arrow to $x_j$. The only way such a sequence of arrows can come from a chain complex over $\mathcal{R}_2$ is if the unique horizontal arrow adjacent to $x_j$ has both an appropriate length (less than $a$) and direction (into $x_j$) as depicted in the figure below. In this case, add a diagonal arrow from $x_i$ to $x_{j\pm 1}$. With this modification, we have $\langle \partial^2 x_i, U^aV^b x_j \rangle=0$.
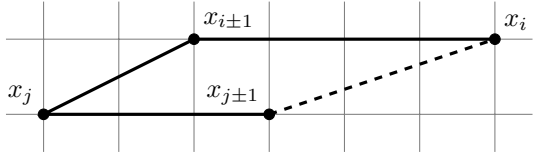
        $$
        \begin{tikzpicture}
            \draw[step=1.0,gray,thin] (0.5,0.5) grid (7.5,2.5);
            \draw[black, very thick] (7,2) -- (3,2) -- (1,1) -- (4,1);
            \draw[black, very thick, dashed] (4,1) -- (7,2);
            \filldraw[black] (7,2) circle (2pt) node[anchor=south west]{$x_i$};
            \filldraw[black] (3,2) circle (2pt) node[anchor=south west]{$x_{i\pm 1}$};
            \filldraw[black] (1,1) circle (2pt) node[anchor=south east]{$x_j$};
            \filldraw[black] (4,1) circle (2pt) node[anchor=south east]{$x_{j\pm 1}$};
        \end{tikzpicture}
        $$
        \item The second arrow is horizontal. Depending on the parity of $j$, the horizontal arrow goes from $x_{j\pm 1}$ to $x_j$ and then there is a diagonal or a vertical arrow from $x_i$ to $x_{j\pm1}$. The only way such a sequence of arrows can come from a chain complex over $\mathcal{R}_2$ is if the unique horizontal arrow adjacent to $x_i$ has both an appropriate length (less than $a$) and direction (out of $x_i$). In this case, add a diagonal arrow from $x_{i \pm 1}$ to $x_j$. With this modification, we have $\langle \partial^2 x_i, U^aV^b x_j \rangle=0$.
    \end{enumerate}
    \item $a=1$: In this case, one of the two arrows must be vertical. The analysis of the situation is completely analogous to the analysis in the horizontal case.
\end{enumerate}
Irrespective of the case we encounter, we draw the prescribed arrow and thus obtain a new free $\mathcal{R}_2$-module with an endomorphism and a distinguished basis $(C, \widetilde{\partial}, B)$ in which $\langle \widetilde{\partial}^2 x_i, U^aV^b x_j \rangle=0$. Note that the distinguished basis has remained the same, there has only been a slight augmentation of the endomorphism. If there exist standard generators $x_i$, $x_j \in B$ in the updated complex such that $\langle \widetilde{\partial}^2 x_i, U^aV^b x_j\rangle = 1$ where $1 \in \{a, b\}$, then we try to repeat the steps above and draw a new arrow. If possible, keep going until $\langle \widetilde{\partial}^2 x_i, U^aV^b x_j\rangle = 0$ for all pairs of standard generators $x_i$, $x_j$ and all $a, b$ with $1\in\{a,b\}$. Once that is achieved, the algorithm terminates and produces a partial realization of $C(a_1, \dots, a_{2n})$. By Proposition \ref{proposition}, the local equivalence class of the standard complex $C(a_1, \dots, a_{2n})$ is algebraically realizable.

\vspace{1em}

It can also happen that, at some point, an arrow cannot be added to the complex as prescribed. In this case, the local equivalence class of the complex is not algebraically realizable. To see this, note that there are no alternative ways of extending $C$ to a chain complex over $\mathcal{R}_2$. In each of the four cases, we can only add diagonal arrows adjacent to the standard generators. Together with the condition $1 \in \{a,b\}$, this implies that the arrows we add in the algorithm \emph{must} be added. 

\vspace{1em}

The output of the algorithm is independent of the order in which we are adding the arrows. If there are many arrows that can be added at a certain stage, we might as well add all of them simultaneously, since they don't interact with each other by the previous paragraph and will have to be added eventually. At this point, it is possible that some of the diagonal arrows that move in both directions by exactly $1$ might have created the need for further arrows to be added. We add those in stage two and proceed similarly for as long as it is needed. In a geometric language, the complex $C(a_1, \dots, a_{2n})$ contains some width 1 tunnels as depicted in Figure \ref{fig:Width 1 tunnel}. Algorithm \ref{algorithm} terminates once all of them have been added the diagonal arrows and the exact order in which this has happened is not important.
\begin{figure}[t]
    \centering
    \begin{tikzpicture}[scale=0.8]
            \draw[step=1.0,gray,thin] (-1.5,-0.5) grid (4.5,5.5);
            \draw[black, very thick] (-1,0) -- (3,0) -- (3,2) -- (-1,2) -- (-1, 5) -- (0,5) -- (0,3) -- (4,3) -- (4,1) -- (2,1);
            \draw[black, very thick, dashed] (-1,0) -- (2,1);
            \draw[black, very thick, dashed] (3,0) -- (4,1);
            \draw[black, very thick, dashed] (3,2) -- (4,3);
            \draw[black, very thick, dashed] (-1,2) -- (0,3);
        \end{tikzpicture}
    \caption{A width 1 tunnel in which the arrows that will be added by Algorithm \ref{algorithm} are dashed. A standard complex can contain many width 1 tunnels and they do not interact with each other.}
    \label{fig:Width 1 tunnel}
\end{figure}

\vspace{1em}

Uniqueness guarantees symmetry. If $C(a_1, \dots, a_{2n})$ is symmetric, the output of the algorithm is symmetric as well.

\vspace{1em}

Finally, let's obtain an upper bound on the number of arrows drawn by the algorithm. There are $2n+1$ generators $x_0, \dots, x_{2n}$ and each one has an associated $\gr_U$. Since $\gr_U(\partial)=-1$, arrows can only be drawn between the generators with $\gr_U$ of different parity. In other words, the number of arrows drawn by the algorithm is bounded by the maximal number of edges in a bipartite graph with $2n+1$ vertices, which is $n(n+1)=n^2+n$. Note that this is a very coarse bound and there will be much fewer arrows added in general.
\end{algorithm}

\begin{proof}[Proof of Theorem \ref{theorem}]
This is the content of Proposition \ref{proposition} and Algorithm \ref{algorithm}. The special case is Corollary \ref{corollary}.
\end{proof}
This completes the classification of algebraically realizable local equivalence classes of knot-like complexes. In particular, it answers a slightly harder version of Question 11.1 of \cite{dai2021more}.

\section{Examples}\label{section:Examples}
We give two demonstrations of our results.
\subsection{The local equivalence class of \texorpdfstring{$C(-1, 1, 2, -1, 1, 2)$}{C(-1, 1, 2, -1, 1, 2)}}
The standard complex $C(-1, 1, 2, -1, 1, 2)$ with a distinguished basis $\{ x_0, \dots, x_6 \}$ is drawn in Figure \ref{fig:Example 1a}. Since we have $\langle \partial^2x_3, U^2Vx_1 \rangle = 1$, the algorithm requires that we add a diagonal arrow from $x_3$ to $x_0$. It is drawn dashed in Figure \ref{fig:Example 1b}. Next, we notice that $\langle \partial^2x_6, UV^2x_4 \rangle = 1$ so the algorithm requires that we add a diagonal arrow from $x_6$ to $x_3$, which is drawn dashed in Figure \ref{fig:Example 1c}.

\begin{figure}[t]
    \begin{subfigure}[b]{0.30\textwidth}
        \centering
        \resizebox{\linewidth}{!}{
        \begin{tikzpicture}[scale=1.0]
            \draw[step=1.0,gray,thin] (0.5,0.5) grid (4.5,3.5);
            \draw[black, very thick] (2, 1) -- (1, 1) -- (1, 2) -- (3, 2) -- (3, 1) -- (4,1) -- (4,3);
            \filldraw[black] (2,1) circle (2pt) node[anchor=north west]{$x_0$};
            \filldraw[black] (1,1) circle (2pt) node[anchor=north west]{$x_1$};
            \filldraw[black] (1,2) circle (2pt) node[anchor=north west]{$x_2$};
            \filldraw[black] (3,2) circle (2pt) node[anchor=north west]{$x_3$};
            \filldraw[black] (3,1) circle (2pt) node[anchor=north west]{$x_4$};
            \filldraw[black] (4,1) circle (2pt) node[anchor=north west]{$x_5$};
            \filldraw[black] (4,3) circle (2pt) node[anchor=north west]{$x_6$};
        \end{tikzpicture}}
        \caption{}
        \label{fig:Example 1a}
    \end{subfigure}
    \begin{subfigure}[b]{0.30\textwidth}
        \centering
        \resizebox{\linewidth}{!}{
        \begin{tikzpicture}[scale=1.0]
            \draw[step=1.0,gray,thin] (0.5,0.5) grid (4.5,3.5);
            \draw[black, very thick] (2, 1) -- (1, 1) -- (1, 2) -- (3, 2) -- (3, 1) -- (4,1) -- (4,3);
            \draw[black, dashed, very thick] (3,2) -- (2,1);
            \filldraw[black] (2,1) circle (2pt) node[anchor=north west]{$x_0$};
            \filldraw[black] (1,1) circle (2pt) node[anchor=north west]{$x_1$};
            \filldraw[black] (1,2) circle (2pt) node[anchor=north west]{$x_2$};
            \filldraw[black] (3,2) circle (2pt) node[anchor=north west]{$x_3$};
            \filldraw[black] (3,1) circle (2pt) node[anchor=north west]{$x_4$};
            \filldraw[black] (4,1) circle (2pt) node[anchor=north west]{$x_5$};
            \filldraw[black] (4,3) circle (2pt) node[anchor=north west]{$x_6$};
        \end{tikzpicture}}
        \caption{}
        \label{fig:Example 1b}
    \end{subfigure}
    \begin{subfigure}[b]{0.30\textwidth}
        \centering
        \resizebox{\linewidth}{!}{
        \begin{tikzpicture}[scale=1.0]
            \draw[step=1.0,gray,thin] (0.5,0.5) grid (4.5,3.5);
            \draw[black, very thick] (2, 1) -- (1, 1) -- (1, 2) -- (3, 2) -- (3, 1) -- (4,1) -- (4,3);
            \draw[black, very thick] (3,2) -- (2,1);
            \draw[black, dashed, very thick] (3,2) -- (4,3);
            \filldraw[black] (2,1) circle (2pt) node[anchor=north west]{$x_0$};
            \filldraw[black] (1,1) circle (2pt) node[anchor=north west]{$x_1$};
            \filldraw[black] (1,2) circle (2pt) node[anchor=north west]{$x_2$};
            \filldraw[black] (3,2) circle (2pt) node[anchor=north west]{$x_3$};
            \filldraw[black] (3,1) circle (2pt) node[anchor=north west]{$x_4$};
            \filldraw[black] (4,1) circle (2pt) node[anchor=north west]{$x_5$};
            \filldraw[black] (4,3) circle (2pt) node[anchor=north west]{$x_6$};
        \end{tikzpicture}}
        \caption{}
        \label{fig:Example 1c}
    \end{subfigure}
    \caption{The algorithm terminates after adding $2$ arrows. Since $\langle \partial^2x_6, U^3Vx_2 \rangle = 1$ at the final stage, the local equivalence class of the starting complex $C(-1, 1, 2, -1, 1, 2)$ is \emph{not} algebraically realizable.}
\end{figure}
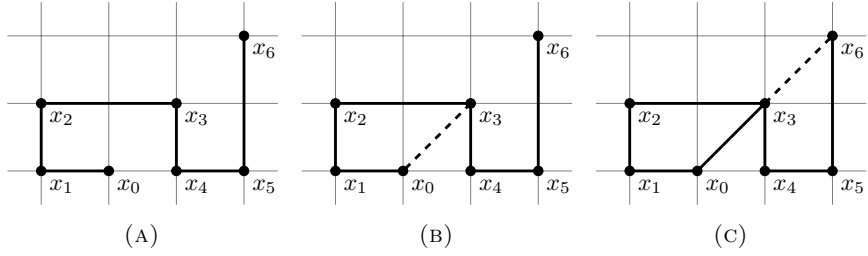

Continuing with the process, we notice that we have $\langle \partial^2x_6, U^3Vx_2 \rangle = 1$ in the updated complex. However, since there is no outgoing horizontal arrow from $x_6$, the algorithm stops and we conclude that the complex $C(-1, 1, 2, -1, 1, 2)$ is not partially realizable and hence the local equivalence class of $C(-1, 1, 2, -1, 1, 2)$ is not algebraically realizable.

\subsection{The local equivalence class of \texorpdfstring{$C(-1, 1, 2, -1, 1, 3)$}{C(-1, 1, 2, -1, 1,3)}}
The standard complex $C(-1, 1, 2, -1, 1, 3)$ with a distinguished basis $\{ x_0, \dots, x_6 \}$ is drawn in Figure \ref{fig:Example 2a}. Since we have $\langle \partial^2x_3, U^2Vx_1 \rangle = 1$, the algorithm requires that we add a diagonal arrow from $x_3$ to $x_0$. It is drawn dashed in Figure \ref{fig:Example 2b}. Next, we notice that $\langle \partial^2x_6, UV^3x_4 \rangle = 1$ so the algorithm requires that we add a diagonal arrow from $x_6$ to $x_3$, which is drawn dashed in Figure \ref{fig:Example 2c}.

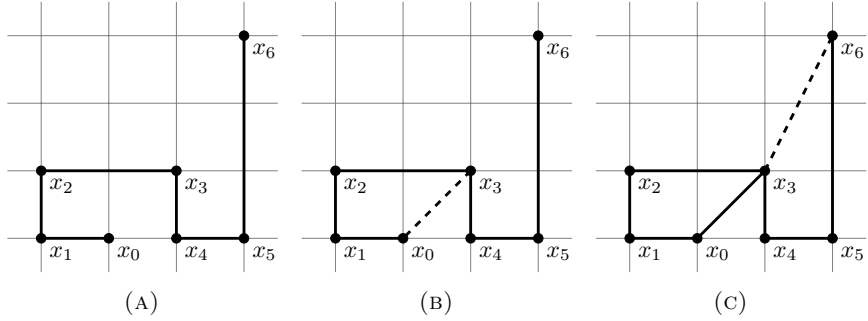
\begin{figure}[t]
    \begin{subfigure}[b]{0.30\textwidth}
        \centering
        \resizebox{\linewidth}{!}{
        \begin{tikzpicture}[scale=1.0]
            \draw[step=1.0,gray,thin] (0.5,0.5) grid (4.5,4.5);
            \draw[black, very thick] (2,1) -- (1, 1) -- (1, 2) -- (3, 2) -- (3, 1) -- (4,1) --(4,4);
            \filldraw[black] (2,1) circle (2pt) node[anchor=north west]{$x_0$};
            \filldraw[black] (1,1) circle (2pt) node[anchor=north west]{$x_1$};
            \filldraw[black] (1,2) circle (2pt) node[anchor=north west]{$x_2$};
            \filldraw[black] (3,2) circle (2pt) node[anchor=north west]{$x_3$};
            \filldraw[black] (3,1) circle (2pt) node[anchor=north west]{$x_4$};
            \filldraw[black] (4,1) circle (2pt) node[anchor=north west]{$x_5$};
            \filldraw[black] (4,4) circle (2pt) node[anchor=north west]{$x_6$};
        \end{tikzpicture}}
        \caption{}
        \label{fig:Example 2a}
    \end{subfigure}
    \begin{subfigure}[b]{0.30\textwidth}
        \centering
        \resizebox{\linewidth}{!}{
        \begin{tikzpicture}[scale=1.0]
            \draw[step=1.0,gray,thin] (0.5,0.5) grid (4.5,4.5);
            \draw[black, very thick] (2, 1) -- (1, 1) -- (1, 2) -- (3, 2) -- (3, 1) -- (4,1) -- (4,4);
            \draw[black, dashed, very thick] (3,2) -- (2,1);
            \filldraw[black] (2,1) circle (2pt) node[anchor=north west]{$x_0$};
            \filldraw[black] (1,1) circle (2pt) node[anchor=north west]{$x_1$};
            \filldraw[black] (1,2) circle (2pt) node[anchor=north west]{$x_2$};
            \filldraw[black] (3,2) circle (2pt) node[anchor=north west]{$x_3$};
            \filldraw[black] (3,1) circle (2pt) node[anchor=north west]{$x_4$};
            \filldraw[black] (4,1) circle (2pt) node[anchor=north west]{$x_5$};
            \filldraw[black] (4,4) circle (2pt) node[anchor=north west]{$x_6$};
        \end{tikzpicture}}
        \caption{}
        \label{fig:Example 2b}
    \end{subfigure}
    \begin{subfigure}[b]{0.30\textwidth}
        \centering
        \resizebox{\linewidth}{!}{
        \begin{tikzpicture}[scale=1.0]
            \draw[step=1.0,gray,thin] (0.5,0.5) grid (4.5,4.5);
            \draw[black, very thick] (2, 1) -- (1, 1) -- (1, 2) -- (3, 2) -- (3, 1) -- (4,1) -- (4,4);
            \draw[black, very thick] (3,2) -- (2,1);
            \draw[black, dashed, very thick] (3,2) -- (4,4);
            \filldraw[black] (2,1) circle (2pt) node[anchor=north west]{$x_0$};
            \filldraw[black] (1,1) circle (2pt) node[anchor=north west]{$x_1$};
            \filldraw[black] (1,2) circle (2pt) node[anchor=north west]{$x_2$};
            \filldraw[black] (3,2) circle (2pt) node[anchor=north west]{$x_3$};
            \filldraw[black] (3,1) circle (2pt) node[anchor=north west]{$x_4$};
            \filldraw[black] (4,1) circle (2pt) node[anchor=north west]{$x_5$};
            \filldraw[black] (4,4) circle (2pt) node[anchor=north west]{$x_6$};
        \end{tikzpicture}}
        \caption{}
        \label{fig:Example 2c}
    \end{subfigure}
    \caption{The algorithm, whose stages are depicted in (\textsc{a}), (\textsc{b}), and (\textsc{c}), terminates after adding $2$ arrows. Since $\partial^2=0$ over $\frac{\F[U,V]}{(U^2V^2)}$ at the final stage, the local equivalence class of the starting complex $C(-1, 1, 2, -1, 1, 3)$ is algebraically realizable.}
\end{figure}

At this point we can verify that $\partial^2=0$ over $\mathcal{R}_2$. This means that the algorithm stops and outputs the result. The complex $C(-1, 1, 2, -1, 1, 3)$ is partially realizable so the local equivalence class of $C(-1, 1, 2, -1, 1, 3)$ is algebraically realizable. A $\Z\oplus\Z$ graded chain complex over $\F[U,V]$ with the correct homology realizing it is constructed using Lemma \ref{lemma:partially real->extended partially real} and Lemma \ref{lemma:extended partially real->local class real} and drawn in Figure \ref{fig:Realization of C(-1, 1, 2, -1, 1, 3)}.

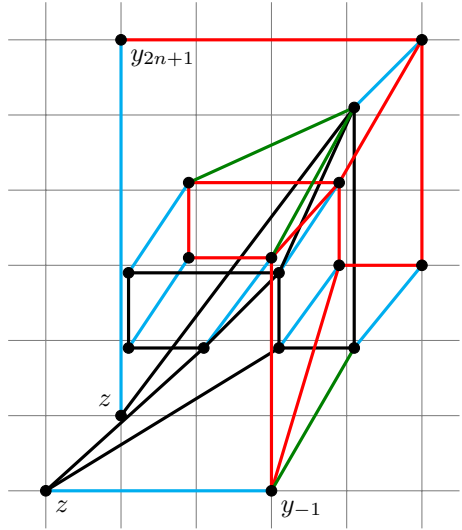
\begin{figure}[t]
    \centering
    \begin{tikzpicture}[scale=1.0]
            \def\a{0.1}
            
            \draw[step=1.0,gray,thin] (-0.5,-1.5) grid (5.5,5.5);
            \draw[cyan, very thick] (2+\a,1-\a) -- (3,2+\a);
            \draw[cyan, very thick] (1+\a,1-\a) -- (2-\a,2+\a);
            \draw[cyan, very thick] (1+\a,2-\a) -- (2-\a,3+\a);
            \draw[cyan, very thick] (3+\a,2-\a) -- (4-\a,3+\a);
            \draw[cyan, very thick] (3+\a,1-\a) -- (4-\a,2);
            \draw[cyan, very thick] (4+\a,1-\a) -- (5,2);
            \draw[cyan, very thick] (4+\a,4+\a) -- (5,5);
            \draw[cyan, very thick] (3, -1) -- (0, -1); %the special blue
            \draw[cyan, very thick] (1,5) -- (1,0);     %the special blue

            \draw[black, very thick] (0,-1) -- (2+\a,1-\a);
            \draw[red, very thick] (5,5) -- (1,5);
            \draw[black, very thick] (1,0) -- (4+\a,4+\a);
            
            \draw[black, very thick] (2+\a, 1-\a) -- (1+\a, 1-\a) -- (1+\a, 2-\a) -- (3+\a, 2-\a) -- (3+\a, 1-\a) -- (4+\a,1-\a) -- (4+\a,4+\a);
            \draw[black, very thick] (3+\a,2-\a) -- (2+\a,1-\a);
            \draw[black, very thick] (3+\a,2-\a) -- (4+\a,4+\a);
            \draw[black, very thick] (0, -1) -- (3+\a, 1-\a);

            \draw[red, very thick] (3, -1) -- (3, 2+\a) -- (2-\a, 2+\a) -- (2-\a, 3+\a) -- (4-\a, 3+\a) -- (4-\a, 2) -- (5,2) -- (5,5);
            \draw[red, very thick] (3, -1) -- (4-\a, 2);
            \draw[red, very thick] (4-\a,3+\a) -- (3,2+\a);
            \draw[red, very thick] (4-\a,3+\a) -- (5,5);

            \draw[dark green, very thick] (4+\a,4+\a) -- (2-\a,3+\a);
            \draw[dark green, very thick] (4+\a,4+\a) -- (3,2+\a);
            \draw[dark green, very thick] (3,-1) -- (4+\a,1-\a);

            \filldraw[black] (2+\a,1-\a) circle (2pt);
            \filldraw[black] (1+\a,1-\a) circle (2pt);
            \filldraw[black] (1+\a,2-\a) circle (2pt);
            \filldraw[black] (3+\a,2-\a) circle (2pt);
            \filldraw[black] (3+\a,1-\a) circle (2pt);
            \filldraw[black] (4+\a,1-\a) circle (2pt);
            \filldraw[black] (4+\a,4+\a) circle (2pt);

            \filldraw[black] (3,2+\a) circle (2pt);
            \filldraw[black] (2-\a,2+\a) circle (2pt);
            \filldraw[black] (2-\a,3+\a) circle (2pt);
            \filldraw[black] (4-\a,3+\a) circle (2pt);
            \filldraw[black] (4-\a,2) circle (2pt);
            \filldraw[black] (5,2) circle (2pt);
            \filldraw[black] (5,5) circle (2pt);

            \filldraw[black] (1,0) circle (2pt) node[anchor =  south east]{$z$};
            \filldraw[black] (0,-1) circle (2pt) node[anchor = north west]{$z$};
            \filldraw[black] (3,-1) circle (2pt) node[anchor = north west]{$y_{-1}$};
            \filldraw[black] (1,5) circle (2pt) node[anchor = north west]{$y_{2n+1}$};
            
        \end{tikzpicture}
    \caption{Algebraic realization of the local equivalence class of $C(-1, 1, 2, -1, 1, 3)$.}
    \label{fig:Realization of C(-1, 1, 2, -1, 1, 3)}
\end{figure}

\bibliography{mybib.bib}
\bibliographystyle{alpha}

\end{document}